\documentclass[leqno,11pt]{amsart}

\usepackage{amssymb}
\usepackage{amsmath}
\usepackage{enumerate}
\usepackage[pdftex]{graphicx}
\usepackage{graphicx, color}

\def\R{\mathbb{R}}

\DeclareMathOperator{\sech}{sech}

\newtheorem{theorem}{Theorem}[section]
\newtheorem{proposition}[theorem]{Proposition}
\newtheorem{corollary}[theorem]{Corollary}

\theoremstyle{definition}
\newtheorem{definition}[theorem]{Definition}
\newtheorem{remark}[theorem]{Remark}

\numberwithin{equation}{section}

\begin{document}

\title[Rotational surfaces with prescribed curvatures]{Rotational surfaces \\ with prescribed curvatures }

\author[P. Carretero]{Paula Carretero}
\address{Departamento de Matem\'aticas \\
	Universidad de Ja\'{e}n \\
	23071 Ja\'{e}n, Spain.}
\email{pch00005@red.ujaen.es}

\author[I. Castro]{Ildefonso Castro}
\address{Departamento de Matem\'{a}ticas \\
	Universidad de Ja\'{e}n \\
	23071 Ja\'{e}n, Spain and IMAG, Instituto de Matem\'aticas de la Universidad de Granada.}
\email{icastro@ujaen.es}

\thanks{The second author is partially supported by State Research Agency and European
	Regional Development Fund via the grants no. PID 2020-117868GB-I00 and PID 2022-142559NB-I00, and Maria de Maeztu Excellence Unit IMAG CEX2020-001105-M funded by MCIN/AEI/10.13039/ 501100011033/
}



\date{}

\begin{abstract}
	We solve the problem of prescribing different types of curvatures (principal, mean or Gaussian) on rotational surfaces in terms of arbitrary continuous functions depending on the distance from the surface to the axis of revolution. In this line, we get the complete explicit classification of the rotational surfaces with mean or Gauss curvature inversely proportional to the distance from the surface to the axis of revolution. We also provide new uniqueness results on some well known surfaces, such as the catenoid or the torus of revolution,
	and others less well known but equally interesting for their physical applications, such as the Mylar balloon or the Flamm's paraboloid.
\end{abstract}

\maketitle

\section{Introduction}

The rotational surfaces, also known as surfaces of revolution or rotation surfaces, are probably the most studied surfaces in Euclidean 3-space in the context of classical differential geometry. One of the main reasons for this interest could be that any such surface can exhibit interesting geometric properties, since the geometry of the surface is directly affected by the geometry of the generatrix curve. For instance, the meridians and the parallels of a surface of revolution are its curvature lines. Outstanding examples of important results in this setting are the classical theorems of Euler (\cite{E44}), Delaunay (\cite{D41}) and Darboux (\cite{D90}), who, in 1744, 1841 and 1890, classified  minimal, constant mean curvature and constant Gauss curvature rotational surfaces, respectively.

K.~Kentmotsu was likely the first one to study rotational surfaces in Euclidean 3-space with prescribed (arbitrary) mean curvature. Specifically, he showed in \cite{Ke80} that a generating curve of a surface of revolution satisfies a nonlinear differential equation describing the mean curvature. He solved this equation by an elementary method which explicitly calculates solutions by generalized
Fresnel's integrals which involve the mean curvature. 
In this way, for a given continuous function $H(s)$, he constructed a 3-parameter family
of surfaces of revolution admitting $H(s)$ as the mean curvature ($s$ being the arc parameter of the generating curve).  In particular, he easily described all complete surfaces of revolution with constant mean curvature in a way different from that in Delaunay \cite{D41}. In addition, he also considered in \cite{Ke80} the case $H(s)=s/2$.

Recently, there has been a renewed interest in studying (hyper)surfaces with prescribed curvature (either mean curvature or Gaussian curvature or a linear combination of both) in terms of a prescribed function of its Gauss map (see, for instance,  \cite{BGM20a}, \cite{BGM20b}, \cite{BO22}, \cite{BO23}). 
By means of a phase plane analysis and under moderate assumptions
on the prescribed function, the authors provide some classifications theorems that generalize some important results in the theory of constant mean and Gauss curvature surfaces, linear Weingarten surfaces and self-translating solitons of the mean curvature flow.

Based on the notion of {\em geometric linear momentum} of a plane curve (see \cite{CCI16} and \cite{CCIs20}), we made a contribution to the study of rotational Weingarten surfaces in Euclidean 3-space (see \cite{CC22}) reducing any type of Weingarten condition (i.e.\ a functional relation between the principal curvatures, see \cite{W61} or \cite{Ch45}) on a rotational surface to a first order ordinary differential equation on the momentum of the generating curve. 
Making use of a similar technique, we now deal with the problem of prescribing different types of curvatures (principal curvatures, mean curvature or Gauss curvature) on rotational surfaces in terms of arbitrary continuous functions depending on the distance from the surface to the axis of revolution. 
Compared to the above mentioned Kenmotsu's idea, instead of taking an \textit{intrinsic} parameter such as the arc length of the generating curve as the independent variable of the functions to be prescribed, we choose an \textit{extrinsic} natural variable such as the distance from the surface to the axis of revolution and the calculations are equally feasible to carry out. Our ambitious goal is to reach the corresponding rotational surfaces in a closed form, that is, as explicit as possible (in terms of elementary functions) trough quadratures.

Our key tool, the geometric linear momentum $\mathcal K$ of a plane curve, was crucial in \cite{CCI16} and \cite{CCIs20} for the study of plane curves whose curvature depends on the distance to a line (a particular case of a more general problem posed by D.~Singer in \cite{S99}). $\mathcal K$ is a smooth function associated with any plane curve which completely determines it (modulo a family of distinguished translations) by quadratures in a constructive way, according to its relative position with respect to a fixed line (see Corollary \ref{cor:key}). 
In addition, it can be interpreted as an anti-derivative of the curvature $\kappa $ of the curve, when this is expressed as a function of the distance $x$ from the fixed line. Therefore, it seems natural that, when dealing with surfaces of revolution generated by the rotation of a plane curve (the generatrix curve) about a co-planar fixed line (the axis of revolution), the geometric linear momentum of the generatrix with respect to the axis of revolution plays a predominant role in controlling the geometry of the surface of revolution. We summarize it in Corollary \ref{cor:keysurfaces}, showing that any surface of revolution is uniquely determined, up to translations along the axis of revolution, by the geometric linear momentum of the generatrix curve. 

In this paper we will first pay attention to what we will call \textit{Hopf-Kühnel surfaces} (see \cite{K15}). They are introduced in Definition \ref{def:Kuhnel} trough an integral parametrization of their generating curves $\alpha_q$, $q\neq 0$. This family includes some well known surfaces, such as the round sphere ($q=1$) or the catenoid ($q=-1$), and others less well known but equally interesting, such as the Mylar balloon ($q=2$), the onducycloid ($q=1/2$) or the Flamm's paraboloid ($q=-1/2$). See Figure \ref{fig:Kuhnel surfaces}.
According to Paulsen \cite{P94}, the \textit{Mylar balloon} is constructed by taking two identical circular
disks of Mylar\footnote{Mylar is an Americanism meaning a trademark for a polyester made in extremely thin sheets of great tensile strength.}, sewing them together along their boundaries, and then inflating the resulting
object with either air or helium. Surprisingly enough, these balloons turn out not to be
spherical as one might expect based on the well-known fact that the sphere possesses
the maximal volume for a given surface area (see \cite{MO03}).
Geometrically they can be described as the surfaces of revolution generated by a special class of elastic curves called lintearias (see Section \ref{Sect2}). The \textit{onducycloid} is defined in \cite{CC22} as the surface of revolution generated by rotation of the cycloid around its base.
The \textit{Flamm's paraboloid} is useful for visualizing the spatial curvature of the Schwarzschild metric and has the property that distances measured within it match distances in the Schwarzschild metric (see, for example, \cite{Kr60}). 
In cylindrical coordinates $(r,\theta,z)$ of $\R^3$, it is given by $z(r)=2\sqrt{r_S(r-r_S)}$, where $r_S$  is the Schwarzschild radius of the massive body.
We are also interested in the \textit{elasticoids}, defined in \cite[Section 3.5]{CC22} as the rotational surfaces generated by the rotation of an elastic curve around its directrix (see Figure \ref{fig:Elasticoids}). 

We characterize in Theorem \ref{Th:Kuhnel} the Hopf-Kühnel surfaces as the only rotational surfaces (apart from the plane) satisfying the linear Weingarten relation $k_{\text m}=q\, k_{\text p}$, $q\neq0$, where $k_{\text m}$ (resp.\ $k_{\text p}$) stands for the principal curvature along meridians (resp.\ parallels).
In this way, not only we generalize \cite[Theorem 6.5]{MO03} on the Mylar balloon corresponding to $q=2$, but we also provide a much more explicit and specific version of \cite[Theorem 3.2]{LP20}. We also emphasize that putting $q=-1$ in Theorem \ref{Th:Kuhnel} we provide a very simple proof of Euler's theorem (cf.\ \cite{E44}) characterizing the catenoid as the only minimal (non flat) rotational surface in Euclidean 3-space.

In Section \ref{Sect3}, we address the problem of whether a rotational surface is uniquely determined if the principal curvatures are prescribed for this kind of surfaces. Theorem \ref{th:Prescribe_kp} gives an affirmative answer when the principal curvature on parallels is prescribed in terms of a function that depends on the distance $x$ from the surface to the axis of revolution. In particular, we show in Corollary \ref{cor:Hopf Kuhnel} that the catenoids are the only surfaces of revolution with $k_{\text p}(x)=c/x^2$, $x \geq c>0$, and the Mylar balloons the only with  $k_{\text p}(x)=c\,x$, $0<x^2\leq1/c$. However, if we prescribe the principal curvature on meridians in terms of a function that depends on the distance from the surface to the axis of revolution, we obtain this time a one-parameter family of rotational surface, see Theorem \ref{th:Prescribe_km}. In this way, we characterize in Corollary \ref{cor:Eleaticoids} the elasticoids as the only rotational surfaces with $k_{\text m}(x)=2ax$, $a>0$. We finish Section \ref{Sect3} classifying in Theorem \ref{Th:exp} the surfaces of revolution with $k_{\text m}(x)=a\, e^x$, $a >0$, since we are able to get an interesting explicit  family including the rotational surface generated by the catenary of equal strength (also known as the grim reaper curve) or the alysoid. See Figures \ref{fig:SuperExpLess1}, \ref{fig:Alysoid} and \ref{fig:SuperExpMoreNeg1}.

Finally, we study in Section \ref{Sect4} the problem of prescribing the mean curvature (resp.\ the Gauss curvature) on a rotational surface in terms of a function $H=H(x)$ (resp.\ $K_G=K_G(x)$) that depends on the distance $x$ from the surface to the axis of revolution. In both generic cases we obtain again uniparametric families and illustrate our constructive method classifying two families with natural hypotheses in this context:
the surfaces of revolution with $H(x)=\mu/x$, $\mu >0$, $x>0$, in Theorem \ref{th:Class_Hinvx}, and those ones with $K_G(x)=\mu/x$, $\mu >0$, $x\neq 0$ in Theorem \ref{th:Class KGinverse}. In the first case, we arrive at a family of surfaces of conical type including the circular cone (see Figures \ref{fig:H05x surfaces}, \ref{fig:Hthetax surfaces} and \ref{fig:Hdeltax surfaces}) and, in the second one, the family of surfaces  generated by the rotation of the cycloids around its base after a translation in the orthogonal direction to its base (see Figures \ref{fig:Cycloids} and \ref{fig:transonducycloidsycloids}). Moreover, we highlight that the mean and Gaussian curvatures on a rotational surface cannot be prefixed simultaneously in terms of arbitrary functions depending on the distance from the surface to the axis of revolution, because they are necessarily related. We exhibit explicitly the constraint in formula \eqref{eq:HlinkKG} and deduce that, once prescribed $H=H(x)$ for a rotational surface, its Gauss curvature $K_G=K_G(x)$ belongs to a one-parameter family of functions. In this line, we get in Theorem \ref{Th:H KG mononomial} an uniqueness result for the Hopf-Kühnel surfaces with interesting consequences such as new characterizations of the Mylar balloons (Corollary \ref{cor:Mylar}), the onducycloids (Corollary \ref{cor:Onducycloids}), the Flamm's paraboloids (Corollary \ref{cor:Flamm paraboloids}) and the tori of revolution (Corollary \ref{cor:tori}).


\section{The geometric linear momentum of a rotational surface}\label{Sect2}

We deal with {\em rotational surfaces}, also called
{\em surfaces of revolution}. They are surfaces globally invariant under the action of any rotation around a fixed line called {\em axis of revolution}.
The rotation of a curve (called {\em generatrix} or {\em profile}) around the fixed line generates a surface of revolution.
The sections of a surface of revolution by half-planes delimited by the axis of revolution, called {\em meridians}, are special generatrices.
The sections by planes perpendicular to the axis are circles called {\em parallels} of the surface.

We denote by $S_\alpha$ the rotational surface in $\R^3$ generated by the rotation around the $z$-axis of a plane curve $\alpha$ in the $xz$-plane. That is, $\alpha $ is the generatrix curve that we can consider parameterized by arc-length, whose parametric equations are given by $x=x(s)$, $y=0$, $z= z(s)$, $s\in I \subseteq \R$. The function $x=x(s)$, $s\in I \subseteq \R$, represents here the signed distance from the point $\alpha(s)$ to the $z$- axis of revolution. Then $S_\alpha $ is parameterized by
$$ S_\alpha \equiv X(s,\theta)=\left(x(s)  \cos \theta, x(s)\sin \theta, z(s)\right), \
(s,\theta)\in I \times (-\pi,\pi).$$
If $X=(x_1,x_2,x_3)$, it is clear that $x_1^2+x_2^2=x^2$ and $x_3=z$.

The rotational surface $S_\alpha $, $\alpha =(x,z)$, is regular if $x\neq 0$. Those points $X(s_0,-)$ such that $x(s_0)=0$ are singular points of $S_\alpha $, unless $\alpha $  meets orthogonally the $z$-axis of revolution  at $\alpha (s_0)=(0,z(s_0))$. If we consider the generatrix curve $\alpha$ as a graph $z=z(x)$, $x \in D \subseteq \R$, we can reparameterize $S_\alpha $ as follows:
$$ S_\alpha \equiv X(x,\theta)=\left(x  \cos \theta, x \sin \theta, z(x) \right), \
(x,\theta)\in D \times (-\pi,\pi).$$
If $(x,z)$ is a generatrix curve of a surface of revolution, so is $(-x,z)$ 
and therefore the rotational surfaces generated by both curves are congruent, that is, $S_{(x,z)}\sim S_{(-x,z)}$. 
In addition, obviously $S_{(x,z+z_0)}\sim S_{(x,z)}$, 
$\forall z_0 \in \R$.

Given any plane curve  $\alpha$ in the $xz$-plane, we introduced in \cite[Section 2]{CC22} the \textit{geometric linear momentum} of $\alpha $ (with respect to the $z$-axis) as a smooth function assuming values in $[-1,1]$ that completely determines it (up to translations in the $z$-direction). 
It is defined by $\mathcal K(s)= \dot z (s)$, where the dot $\, \dot{}\, $ means derivation with respect to the arc parameter $s$. Geometrically, $\mathcal K$ controls the angle of the Frenet frame of the curve with the coordinate axes. 
Moreover, in physical terms, $\mathcal K= \mathcal K (s)$ may be described as the linear momentum (with respect to the $z$-axis) of a particle of unit mass with unit speed and trajectory $\alpha (s)$.
We point out that $\mathcal K$  is well defined, up to the sign, depending on the orientation of $\alpha$.
\begin{remark}\label{re:no ppa}
	If the plane curve $\alpha=(x,z)$ is not necessarily parameterized by arc length, i.e. $\alpha = \alpha (t)$, $t$ being any parameter, one can compute the geometric linear momentum $\mathcal K = \mathcal K (t)$ by means of
	$$
	\mathcal K (t)= \frac{z'(t)}{|\alpha ' (t)|},
	$$
	where $'$ denotes derivation respect to $t$.
\end{remark}
The importance of the geometric linear momentum $\mathcal K$ lies in the fact that it allows to determine by quadratures, in a constructive explicit way, the plane curves $\alpha=(x,z)$ such that its curvature depends on the distance to the $z$-axis, that is, it is given as a function of $x$, i.e.\ $\kappa=\kappa(x)$. In this case, $\mathcal K = \mathcal K (x)$ too and satisfies $\mathcal K'(x)=\kappa (x)$ and
the algorithm to recover the curve $\alpha=(x,z)$ involves the following computations (see \cite[Remark 2]{CC22}):
\begin{enumerate}[\rm (i)]
	\item[\rm (i)] Arc-length parameter $s$ of $\alpha=(x,z) $ in terms of $x$, defined ---up to translations of the parameter--- by the integral:
\begin{equation}\label{eq:s(x)}
	s=s(x)=\int\!\frac{dx}{\sqrt{1-\mathcal K(x)^2}},
\end{equation}
	where $-1<\mathcal K(x)<1$, and
	inverting $s=s(x)$ to get $x=x(s)$. 
	\item[\rm (ii)] $z$-coordinate of the curve ---up to translations along $z$-axis--- by the integral:
\begin{equation}\label{eq:z(s)}
	z=z (s)=\int \! \mathcal K(x(s))\, ds .
\end{equation}
\end{enumerate}
Alternatively, if we eliminate $ds$ in the above integrals, we obtain:	 
\begin{equation}\label{eq:truco}
	z=z(x)= \int \frac{\mathcal K(x)\,dx}{\sqrt{1-\mathcal K(x)^2}}.
\end{equation}	 
Thus we can summarize the determining role of the geometric linear momentum in the next result.
\begin{corollary}\cite[Corollary 1]{CC22}
	\label{cor:key}
	Any plane curve $\alpha =(x,z)$, with $x$ non-constant, is uniquely determined by its geometric linear momentum $\mathcal K$ as a function of its distance to $z$-axis, that is, by $\mathcal K= \mathcal K(x)$. 
	The uniqueness is modulo translations in the $z$-direction.
	Moreover, the curvature of $\alpha $ is given by $\kappa (x)=\mathcal K' (x)$.
\end{corollary}
As we mentioned before, if we translate the generatrix curve $\alpha$ of a rotational surface $S_\alpha$  along $z$-axis, obviously we obtain a surface congruent to $S_\alpha$. An immediate consequence of Corollary \ref{cor:key} is then the following key result for our purposes: 
\begin{corollary}\cite[Corollary 2]{CC22}
	\label{cor:keysurfaces}
	Any rotational surface $S_\alpha$, with generatrix curve $\alpha = (x,z)$, is uniquely determined, up to $z$-translations, by the geometric linear momentum $\mathcal K=\mathcal K (x)$ of its generatrix curve, being $x$ non-constant. 
\end{corollary}
\begin{remark}\label{re:cylinder}
	The only rotational surface excluded in Corollary \ref{cor:keysurfaces} is the {\em right circular cylinder}, corresponding to $x$ being constant. We recall that it is a flat rotational surface and its principal curvatures are $0$ (along generatrices lines) and $1/a$ (along parallel circles), $a>0$ being the radius of the cylinder.
\end{remark}

Making use of Corollary \ref{cor:keysurfaces}, we can list the following characterizations of some simple surfaces of revolution (see \cite[Proposition 1]{CC22}):

\begin{enumerate}
	\item Any horizontal \textit{plane} is uniquely determined by the geometric linear momentum $\mathcal K \equiv 0$.
	\item The \textit{circular cone} with opening $\theta_0 \in (0, \pi/2)$, given by $x_1^2+x_2^2=\cot^2\theta_0\,x_3^2$, is uniquely determined by the geometric linear momentum $\mathcal K \equiv \sin \theta_0$.
	\item The \textit{sphere} of radius $R>0$, given by $x_1^2+x_2^2+x_3^2=R^2$, is uniquely determined by the geometric linear momentum $\mathcal K (x)=x/R$.
	\item The \textit{torus of revolution} with major radius $|a|\neq 0$ and minor radius $R>0$, given by $(\sqrt{x_1^2+x_2^2}-a)^2+x_3^2=R^2$, is uniquely determined by the geometric linear momentum  $\mathcal K (x)=(x-a)/R$.
\end{enumerate}

We can confirm the result established in Corollary \ref{cor:keysurfaces} when we study the geometry of $S_\alpha$ through its first and second fundamental forms, $I$ and $II$, since a direct computation, using that $\kappa (x)=\mathcal K ' (x) $, shows that both can be expressed only in terms of the geometric linear momentum $\mathcal K$ and, of course, the non constant signed distance $x$ from the surface to the axis of revolution:
$$ I\equiv ds^2 + x^2 \, d\theta ^2, \quad II\equiv \mathcal K ' (x) \, ds^2 + x \, \mathcal K (x) \, d\theta^2. $$

Therefore we get the following expressions for the principal curvatures $\kappa_1$ and $\kappa_2$, whose curvature lines are the meridians (m) and the parallels (p) respectively of the rotational surface $S_\alpha$:
\begin{equation}\label{eq:prin curv}
	\kappa_1\equiv k_{\text m}= \mathcal K ' (x), \quad \kappa_2\equiv k_{\text p}= \frac{\mathcal K (x)}{x}.
\end{equation}
Thus, the mean curvature $H$ of $S_\alpha $ is given by
\begin{equation}\label{eq:H}
	2H=\mathcal K ' (x)+\frac{\mathcal K (x)}{x},
\end{equation}
and the Gauss curvature $K_{\text G}$ of $S_\alpha $ is given by
\begin{equation}\label{eq:KGauss}
	K_{\text G}=\frac{\mathcal K (x) \mathcal K ' (x) }{x}.
\end{equation}

On the other hand, Weingarten surfaces are defined by a functional relation between their principal curvatures.
For rotational Weingarten surfaces,
we simply write $$\Phi(k_{\text m}, k_{\text p})=0.$$
Hence, taking into account \eqref{eq:prin curv}, we easily deduce that the above functional relation translates into a first-order differential equation for the geometric linear momentum $\mathcal K=\mathcal K (x)$ determining $S_\alpha$ according to Corollary \ref{cor:keysurfaces}: 
$$ \hat \Phi (x,\mathcal K  (x),\mathcal K ' (x))=0. $$

Now we deal with two interesting families of rotational surfaces that will play an important role along the paper.
\begin{definition}[Hopf-Kühnel surfaces]\label{def:Kuhnel}
	 Consider the rotational surfaces $S_{\alpha_q}$, $q\neq 0$, where $\alpha_q=(x,z)$ is given by
	 $$
	 x= a \cos ^{1/q} t, \ z= \frac{a}{q}\int_0^t \cos ^{1/q} v \, dv, \ t\in (-\pi/2,\pi/2), \ a>0. 
	 $$
	 We call them {\em Hopf-Kuhnel surfaces} because they were first introduced by Heinz Hopf  \cite{H51} in 1951 and subsequently studied by Wolfgang Kühnel \cite{K15} in 1999 (see also \cite{F93}, \cite{KS05} and \cite{LP20} among others).
\end{definition}
Using Remark \ref{re:no ppa}, it is an exercise to check that the linear geometric momentum of $\alpha_q$ is given by
\begin{equation}\label{eq:K Kuhnel}
	\mathcal K (x)= \frac{x^q}{a^q}.
\end{equation}

In the next result, we determine the surfaces of revolution such that both principal radii of curvature (or, equivalently, both principal curvatures) are proportional at each point.
	
	\begin{theorem}\label{Th:Kuhnel}
		The only rotational surfaces satisfying the linear Weingarten relation $k_{\text m} = q\,  k_{\text p}$, $q\neq 0$, are the plane and the Hopf-Kühnel surfaces $S_{\alpha_q}$.
	\end{theorem}
	
	\begin{proof}
		First, the plane trivially satisfies $k_{\text m} = q\,  k_{\text p}$, for any $q\neq 0$. Using \eqref{eq:K Kuhnel} in \eqref{eq:prin curv}, it is easy to check that $S_{\alpha_q}$ verifies precisely $k_{\text m} = q\,  k_{\text p}$.
		
		On the other hand, from \eqref{eq:prin curv} 
		the~linear Weingarten relation $k_{\text m}=q\,  k_{\text p}$ translates into the separable~o.d.e.
		$$\mathcal K ' (x)=q\, \mathcal K(x)/x. $$
		Its constant solution $\mathcal K \equiv 0$ leads to the plane.
		Its non-constant solution is given by $\mathcal K (x)= c \, x^q$, $c>0$. Taking $a=1/c^{1/q}$, we arrive at \eqref{eq:K Kuhnel}. The proof follows from Corollary \ref{cor:keysurfaces}.
	\end{proof}	

The family of Hopf-Kühnel surfaces includes very interesting rotational surfaces for particular values of the constant $q\neq 0$. As a consequence of Corollary \ref{cor:keysurfaces} and \eqref{eq:K Kuhnel}, the following surfaces are uniquely determined, up to translations in $z$-direction, by the geometric linear momentum indicated in each case (see also \cite[Proposition 1]{CC22}):
	\begin{enumerate}
		\item $q=1$: The \textit{sphere} of radius $R>0$, given by $x_1^2+x_2^2+x_3^2=R^2$, by $\mathcal K (x)=x/R$.
		\item $q=-1$: The \textit{catenoid} of chord $a\! >\!0$, given by $x_1^2+x_2^2 \!=\! a^2 \cosh^2 (x_3/a)$, by $\mathcal K (x)=a/x$.
	\item $q=2$: The \textit{Mylar balloon} of inflated radius $r>0$ (see \cite{MO03}),  by $\mathcal K (x)=x^2/r^2$.
	\item $q=1/2$: The \textit{onducycloid} of radius $R>0$, defined as the surface generated by rotation of the cycloid around its base, by $\mathcal K (x)= \sqrt x/\sqrt{2R}$.
	\item $q=-1/2$: The \textit{Flamm's paraboloid} of Schwarzchild radius $r_S>0$ (see \cite{Kr60}), defined as the surface generated by the rotation of the parabola with vertex $(r_S,0)$ and focus $(2r_S,0)$, around its directrix line, given by $x_3^2=4r_S(\sqrt{x_1^2+x_2^2}-r_S)$, is uniquely determined by the geometric linear momentum  $\mathcal K (x)= \sqrt{r_S}/\sqrt{x}$.
	\end{enumerate}
The above Hopf-Kühnel surfaces are depicted in Figure \ref{fig:Kuhnel surfaces}.

\begin{figure}[h!]
	\begin{center}
		\includegraphics[height=4.8cm]{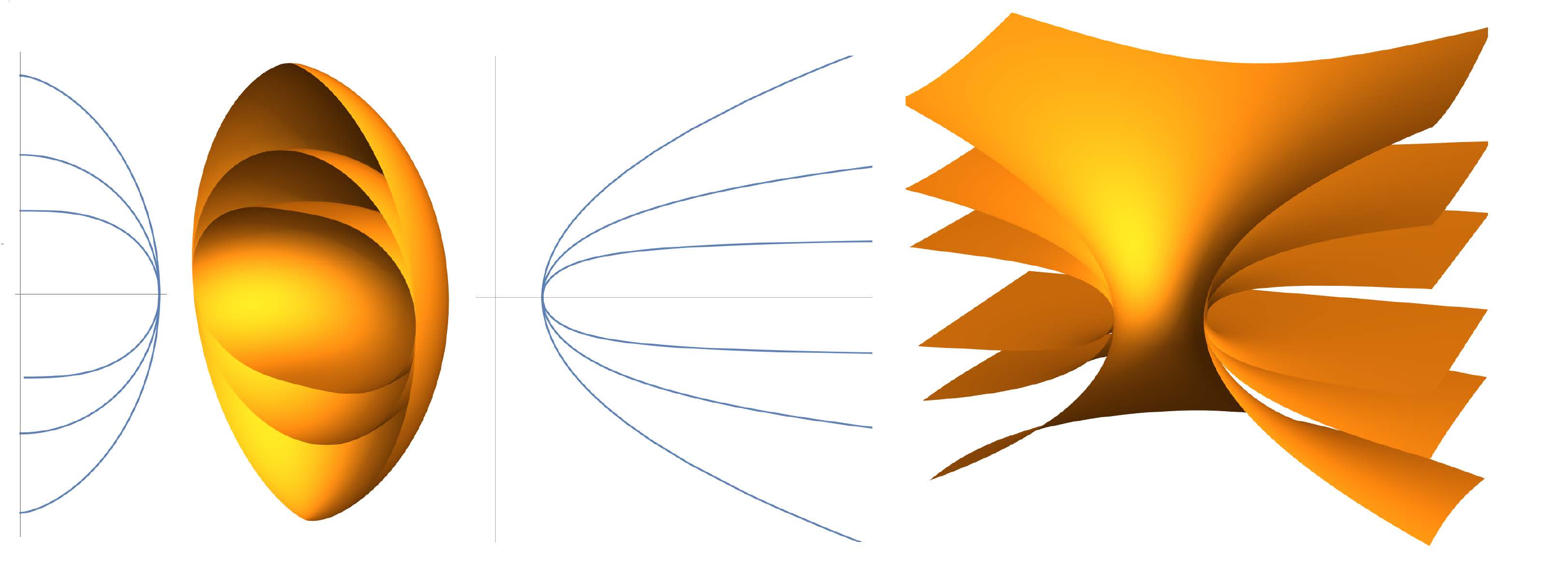}
		\caption{Open views of some Hopf-Kühnel surfaces $S_{\alpha_q}$. From top to bottom, on the left, for $q=1/2$ (onducycloid), $q=1$ (sphere) and $q=2$ (Mylar balloon); on the right, for $q=-1/2$ (Flamm's paraboloid), $q=-1$ (catenoid) and $q=-2$.}
		\label{fig:Kuhnel surfaces}
	\end{center}
\end{figure}

\begin{remark}\label{Re:Kuhnel}
	Theorem \ref{Th:Kuhnel} generalizes \cite[Theorem 6.5]{MO03} on the Mylar balloon corresponding to $q=2$. We also point out that putting $q=-1$ in Theorem \ref{Th:Kuhnel} we recover Euler's theorem (cf.\ \cite{E44}) characterizing the plane and the catenoid as the only minimal rotational surfaces in Euclidean 3-space.
\end{remark}

\medskip

\begin{definition}[Elasticoids]\label{def:elasticoids}
The elastic curves can be defined as those plane curves whose curvature is, at all points, proportional to the distance to a fixed line, called the directrix. 
 We defined in \cite[Section 3.5]{CC22} the \textit{elasticoids} as the rotational surfaces generated by the rotation of an elastic curve around its directrix. 
\end{definition}
With the elastic curve in the $xz$-plane and the directrix as the $z$ axis, the above condition can be written as $\kappa (x) = 2a x$, $a>0$, and then we have a one-parameter family of elastic curves $e_k$, $k>-1$, determined, up to $z$-translations, by the geometric linear momenta
\begin{equation}\label{eq:K elasticoids}
	\mathcal K (x)= ax^2 - k, \ k>-1,
\end{equation}
since $\mathcal K (x)^2<1$.
Using Corollary \ref{cor:keysurfaces}, the elasticoids $S_{e_k}$, $k>-1$, are uniquely determined, up to translations along $z$-axis, by the geometric linear momenta given in \eqref{eq:K elasticoids}. 
Following \cite{F93} or \cite{S08}, we can distinguish seven types of elasticoids according to the seven types of elastic curves depending on the possible ranges of values of the \textit{modulus} $k\in (-1,+\infty) $ of the elliptic functions which appear in the parametrizations of the elastic curves generating the elasticoids. See Figure \ref{fig:Elasticoids}.

\begin{figure}[h!]
	\begin{center}
		\includegraphics[height=7cm]{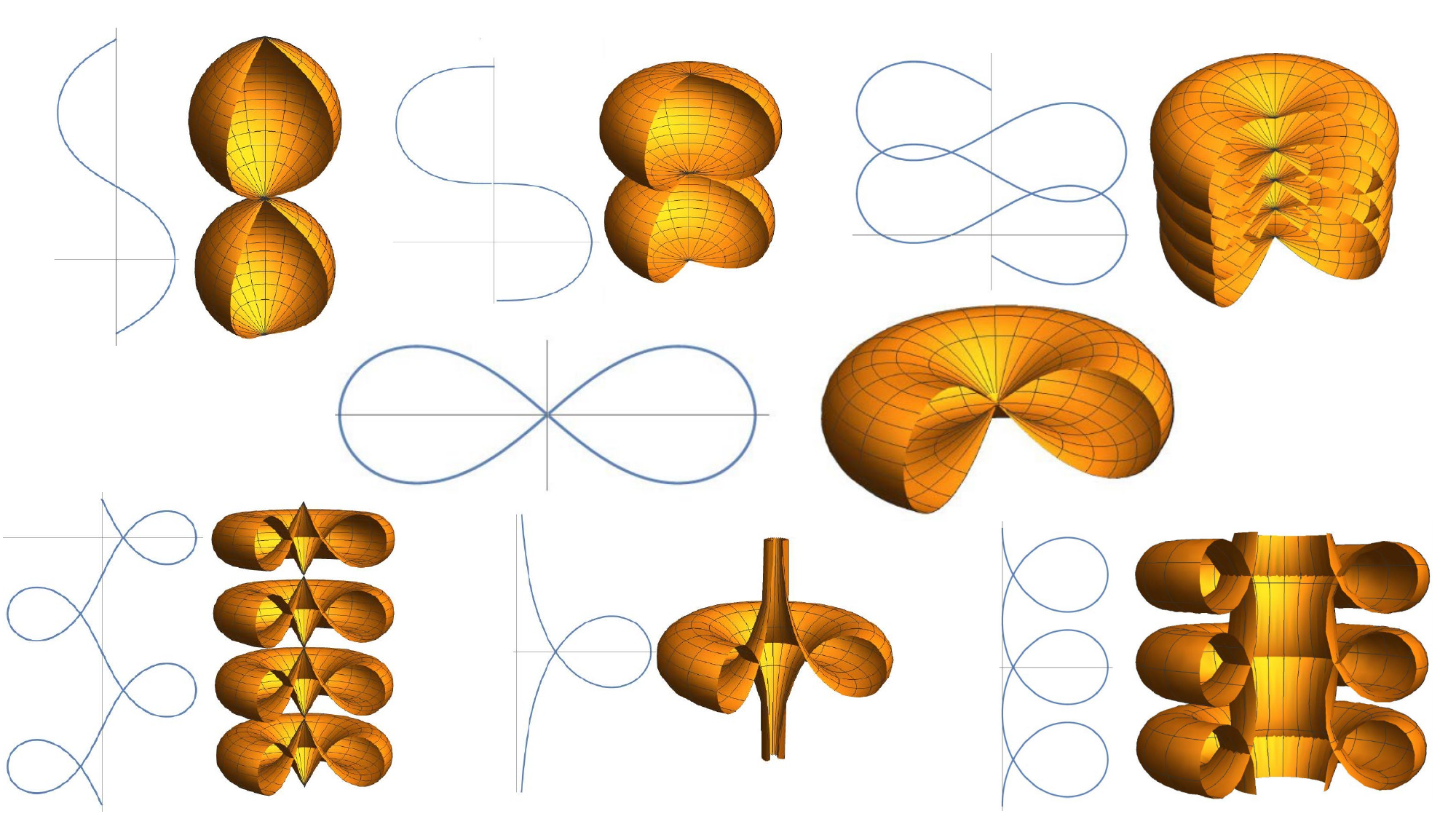}
		\caption{Open views of the elasticoids $S_{e_k}$, $k>-1$: surfaces of revolution generated by [first row, from left to right]
			the pseudo-sinusoids ($-1<k<0$), the right \textit{lintearia} ($k=0$), the elastic curves $e_k$ with $0<k<k_1$,
		[second row] the pseudolemniscate ($k=k_1:=0.65222\dots$),
	[third row, from left to right] the elastic curves $e_k$ with $k_1<k<1$, the convict curve ($k=1$), the pseudotrochoids ($k>1$).}
		\label{fig:Elasticoids}
	\end{center}
\end{figure}

Using \eqref{eq:prin curv} and \eqref{eq:K elasticoids}, it is easy to check that
the elasticoid with null modulus, i.e.\ $k=0$, satisfies $k_{\text m}=2 k_{\text p}$ and so the elasticoid generated by the lintearia is nothing but the Mylar balloon (see above). Otherwise, in \cite[Theorem 3]{CC22} we provided the following uniqueness result for this family of rotational surfaces:
{\em The only rotational surfaces verifying $k_{\text m}^2-2 k_{\text p}\, k_{\text m} + \mu=0$, $\mu \neq 0$, are the sphere of radius $R=1/\mu$ ($\mu >0$) and the elasticoids with nonzero modulus.}


\section{Prescribing principal curvatures on a rotational surface}\label{Sect3}

In this section we aim to address the problem of whether a rotational surface is (uniquely) determined if the principal curvatures are prescribed for this kind of surfaces.

\subsection{On parallels}

Our first result gives an affirmative answer when the principal curvature on parallels is prescribed in terms of a function that depends on the distance from the surface to the axis of revolution. It is a direct consequence of Corollary \ref{cor:keysurfaces} and \eqref{eq:prin curv}, bearing also in mind that the geometric linear momentum takes values in $[-1,1]$.
\begin{theorem}
	\label{th:Prescribe_kp}
	Let $p=p(x)$, $x\in D \subseteq \R $, be a continuous function such that $-1 \leq x\, p(x) \leq 1$.
	Then there exists only one rotational surface whose principal curvature on parallels is $p(x)$, 
	$x$ being the (non constant) distance from the surface to the axis of revolution.
	The surface is uniquely determined, up to translations along $z$-axis, by the geometric linear momentum of its generatrix curve given by
	\begin{equation}\label{eq:mom-kp}
		 \mathcal K (x) =x \, p(x).
	\end{equation}
\end{theorem}
Applying Theorem \ref{th:Prescribe_kp}, looking at \eqref{eq:mom-kp} and Section \ref{Sect2}, we can easily conclude the following simple characterizations of standard surfaces of revolution: if we prescribe $p\equiv 0$, we recover the plane; setting $p \equiv p_0$, $p_0>0$, a priori, we detect the sphere of radius $1/p_0$ and, if we admit $x$ to be constant, the right circular cylinder of the same radius (see Remark \ref{re:cylinder}); prefixing $p(x)=c/x$, $|c|<1$, we retrieve the cone with opening $\theta_0 =\arcsin c$. We emphasize the next new characterization of the Hopf-Kühnel surfaces defined in Definition \ref{def:Kuhnel}, coming from Theorem \ref{th:Prescribe_kp} and \eqref{eq:K Kuhnel}.
\begin{corollary}\label{cor:Hopf Kuhnel}
	Let $p(x)=c\, x^m$, $c>0$, $m\neq -1$, $0<x^{m+1} \leq 1/c$. The only rotational surface whose principal curvature on parallels is $p(x)$, $x$ being the (non constant) distance from the surface to the axis of revolution, is the Hopf-Kühnel surface $S_{\alpha_{q}}$, with $q=m+1$. 
	
	In particular, the catenoid (resp.\ the Mylar balloon) is the unique surface of revolution with $p(x)=c/x^2$, $x \geq c>0$ (resp.\ with $p(x)=c\,x$, $0<x^2\leq1/c$).
\end{corollary}

\subsection{On meridians}

As a new direct consequence of Corollary \ref{cor:keysurfaces} and \eqref{eq:prin curv}, we deduce that if we prescribe the principal curvature on meridians in terms of a function that depends on the distance from the surface to the axis of revolution, we obtain this time a one-parameter family of rotational surfaces. 

\begin{theorem}
	\label{th:Prescribe_km}
	Let $k=k(x)$, $x\in D \subseteq \R $, be a continuous function.
	Then there exists a one-parameter family of rotational surfaces whose principal curvature on meridians is $k(x)$, 
	$x$ being the (non constant) distance from the surface to the axis of revolution.
	The surfaces in the family are uniquely determined, up to translations along $z$-axis, by the geometric linear momenta of their generatrix curves given by
	\begin{equation}\label{eq:mom-km}
		\mathcal K (x) =  \int \!   k(x) dx.
	\end{equation}
\end{theorem}

The parameter in the uniparametric family described in Theorem \ref{th:Prescribe_km} comes from the integration constant in \eqref{eq:mom-km}.

Using Theorem \ref{th:Prescribe_km}, taking into account \eqref{eq:mom-km} and Section \ref{Sect2}, we can easily get simple characterizations of the following surfaces of revolution: if we prescribe $k\equiv 0$, we recover the circular cones -including the plane as limit case- and, if we admit $x$ to be constant, the right circular cylinder (see Remark \ref{re:cylinder}); setting $k \equiv k_0$, $k_0>0$, we detect the sphere of radius $1/k_0$ and the tori of revolution with the same minor radius. Finally, using again Theorem \ref{th:Prescribe_km} and \eqref{eq:K elasticoids}, we point out the following new characterization of the elasticoids defined in Definition \ref{def:elasticoids}.

\begin{corollary}\label{cor:Eleaticoids}
	The elasticoids are the only rotational surfaces whose principal curvature on meridians is given by $k(x)=2a x$, $a >0$, $x$ being the (non constant) distance from the surface to the axis of revolution.
\end{corollary}

We finally plan to study the rotational surfaces whose principal curvature on meridians is given by $k(x)=a\, e^x$, $a >0$, $x$ being the (non constant) distance from the surface to the axis of revolution. According to Theorem \ref{th:Prescribe_km}, we are devoted to study their generatrix curves, uniquely determined up to $z$-translations by the following geometric linear momenta:
\begin{equation}\label{eq:K exp}
	\mathcal K (x)=a \, e^x+c, \ a>0, \,  c<1.
\end{equation}
For this purpose, we are going to use the algorithm described by \eqref{eq:s(x)} and \eqref{eq:z(s)} in Section \ref{Sect2}. In this way, we have that
\begin{equation}\label{eq:s(t)}
	s=s(x)=\int \frac{dx}{\sqrt{1-(a \, e^x+c)^2}}=-\int \frac{dt}{\sqrt{P_a^c(t)}}, 
\end{equation}
where $ P_a^c(t)=(1-c^2)\, t^2 -2ac\, t-a^2 $ is the quadratic polynomial resulting from the change of variable  $t=e^{-x}>0$ in \eqref{eq:s(t)}. Its discriminant is $4a^2$. We distinguish the following cases taking into account that $ P_a^c(t)>0$:
\begin{itemize}
	\item Case $c=\sin \beta \in (-1,1)$, $\beta \in (-\pi/2,\pi/2)$:
	
	Then $P_a^\beta(t)=c_\beta^2 \, t^2-2a s_\beta \, t -a^2$, $t > \frac{a}{1-s_\beta}$, where $s_\beta:= \sin \beta$ and $c_\beta := \cos \beta$. After a suitable translation in $s$ if necessary, a long straightforward computation from \eqref{eq:s(t)} leads to $t=a(\cosh(c_\beta s)+s_\beta)/c_\beta^2=e^{-x}$ and so
	\begin{equation}\label{eq:x beta}
		x_a^\beta(s)=\ln \left( \frac{c_\beta^2}{a(\cosh(c_\beta s)+s_\beta)} \right), \ s \in \R.
	\end{equation}
	 Now \eqref{eq:z(s)} says that $z_a^\beta(s)=\int (a e^{x_a^\beta (s)}+s_\beta)ds$ and, using \eqref{eq:x beta}, we conclude that
	 	\begin{equation}\label{eq:z beta}
	 	z_a^\beta(s)=s_\beta \, s + 2 \arctan \left( \frac{e^{c_\beta s}+s_\beta}{c_\beta} \right), \ s \in \R.
	 \end{equation}
 Let $\alpha_a^\beta=(x_a^\beta,z_a^\beta)$ and recall that the corresponding rotational surfaces are denoted by  $S_{\alpha_a^\beta}$, $a>0$, $\beta \in (-\pi/2,\pi/2)$. The particular case $\beta =0$ recovers the curve 
	 $$\alpha_a^0 (s)=\left(\ln\left(\frac{\sech s}{a}\right),2\arctan e^s\right),$$
	 whose intrinsic equation is given by $k(s)=a\, e^{x_a^0(s)}= \sech s$. It is simply the graph $x=\ln (\frac1a \sin z)$, $z>0$, and is known in literature as the \textit{catenary of equal strength}, since it is the shape taken by an inextensible flexible massive wire hanging from two points, when the linear mass density (i.e., in practical terms, the width of the wire) is proportional to the tension (see \cite{F93}). This is the reason why we will call \textit{generalized catenoids of equal strength} the rotational surfaces $S_{\alpha_a^\beta}$, $a>0$, $\beta \in (-\pi/2,\pi/2)$. See Figure \ref{fig:SuperExpLess1}.
	 \begin{figure}[h!]
	 	\begin{center}
	 		\includegraphics[height=9cm]{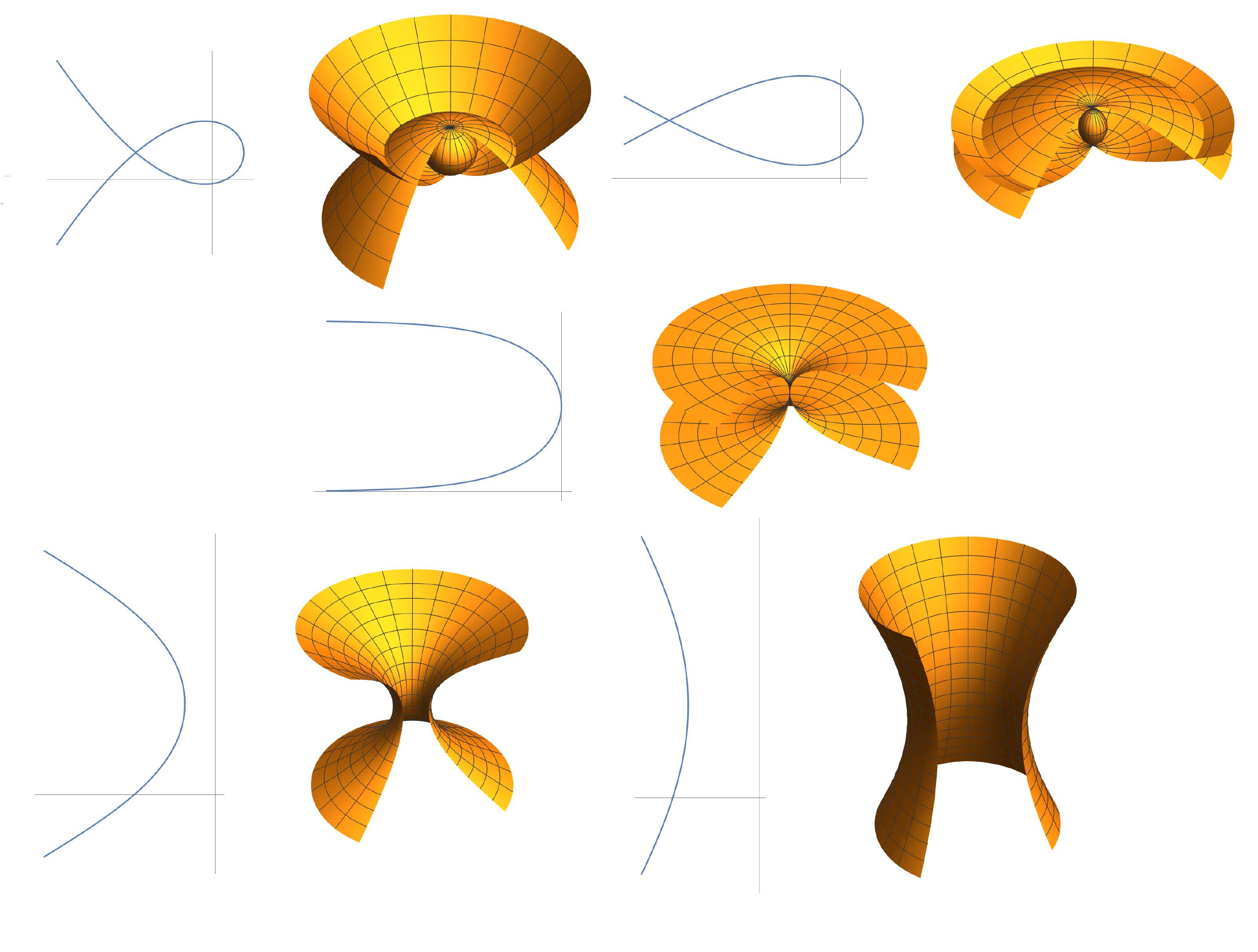}
	 		\caption{Open views of the generalized catenoids of equal strength $S_{\alpha_a^\beta}$, $a=1$, $\beta \in (-\pi/2,\pi/2)$:
	 				 [first row]
	 			with $\beta=-\pi/3$ and  $\beta=-\pi/6$,
	 			[second row] the catenoid of equal strength ($\beta =0$),
	 			[third row] with $\beta=\pi/6$ and $\beta=\pi/3$.}
	 		\label{fig:SuperExpLess1}
	 	\end{center}
	 \end{figure}
	  
	\item Case $c=-1$:
	
	Now $P_a^{-1}=a(2t-a)$, $t>a/2$. Then \eqref{eq:s(t)} implies easily that $t=a(1+s^2)/2=e^{-x}$. Thus:
	\begin{equation}\label{eq:x -1}
		x_a^{-1}(s)=\ln \left( \frac{2}{a(1+s^2)} \right), \ s \in \R,
	\end{equation}
	and \eqref{eq:z(s)} gives that $z_a^{-1}(s)=\int (a e^{x_a^{-1} (s)}-1)ds$. Using \eqref{eq:x -1}, we arrive at
	\begin{equation}\label{eq:z -1}
		z_a^{-1}(s)=2 \arctan s -s, \ s \in \R.
	\end{equation}
The intrinsic equation of the curve $\alpha_a^{-1}=(x_a^{-1},z_a^{-1})$ is given by $k(s)=a\, e^{x_a^{-1}(s)}= 2/(1+s^2)$. This curve was studied studied by Cesàro in 1886 and is just the \textit{alysoid}, from Greek allusion "little chain" (cf.\ \cite{F93}); we will call the \textit{ondualysoid} the corresponding rotational surface $S_{\alpha_a^{-1}}$. See Figure \ref{fig:Alysoid}
	\begin{figure}[h!]
		\begin{center}
			\includegraphics[height=4cm]{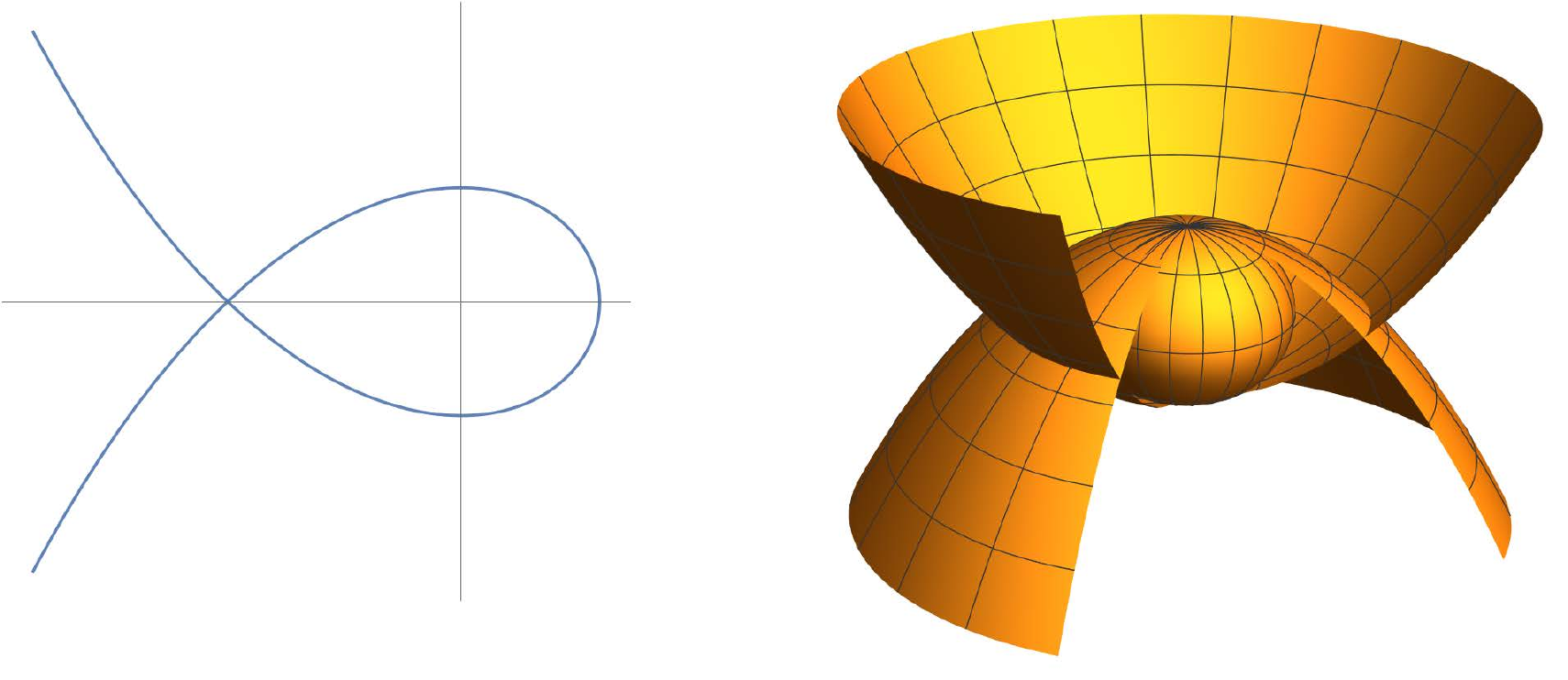}
			\caption{The alysoid and an open view of the ondualysoid $S_{\alpha_a^{-1}}$.}
			\label{fig:Alysoid}
		\end{center}
	\end{figure}
		
	\item Case $c=-\cosh \eta < -1$, $\eta >0$:
	
	Thus $P_a^\eta(t)=-sh_\eta^2 \, t^2 +2a ch_\eta \, t -a^2$, $\frac{a}{ch_\eta +1}<t<\frac{a}{ch_\eta -1}$, with $ch_\eta:=\cosh \eta$ and $sh_\eta := \sinh \eta$. A straightforward computation from \eqref{eq:s(t)} leads to $t=a(ch_\eta-\sin(sh_\eta s))/sh_\eta^2=e^{-x}$ and so
	\begin{equation}\label{eq:x eta}
		x_a^\eta(s)=\ln \left( \frac{sh_\eta^2}{a(ch_\eta-\sin(sh_\eta s))} \right), \ s \in \R. 
	\end{equation}
	Now \eqref{eq:z(s)} implies that $z_a^\eta(s)=\int (a e^{x_a^\eta (s)}-ch_\eta)ds$ and, using \eqref{eq:x eta}, we conclude that
	\begin{equation}\label{eq:z eta}
		z_a^\eta(s)=-ch_\eta \, s - 2 \arctan \left( \dfrac{1-ch_\eta \tan \left( \frac{sh_\eta s}{2} \right)}{sh_\eta} \right), \ s \in \R .
	\end{equation}
	Let $\alpha_a^\eta=(x_a^\eta,z_a^\eta)$. Its picture looks like loops along the vertical direction. We denote the corresponding rotational surfaces by  $S_{\alpha_a^\eta}$, $a>0$, $\eta >0$, and they will be called \textit{loopoids}. See Figure \ref{fig:SuperExpMoreNeg1}
\begin{figure}[h!]
	\begin{center}
		\includegraphics[height=4.2cm]{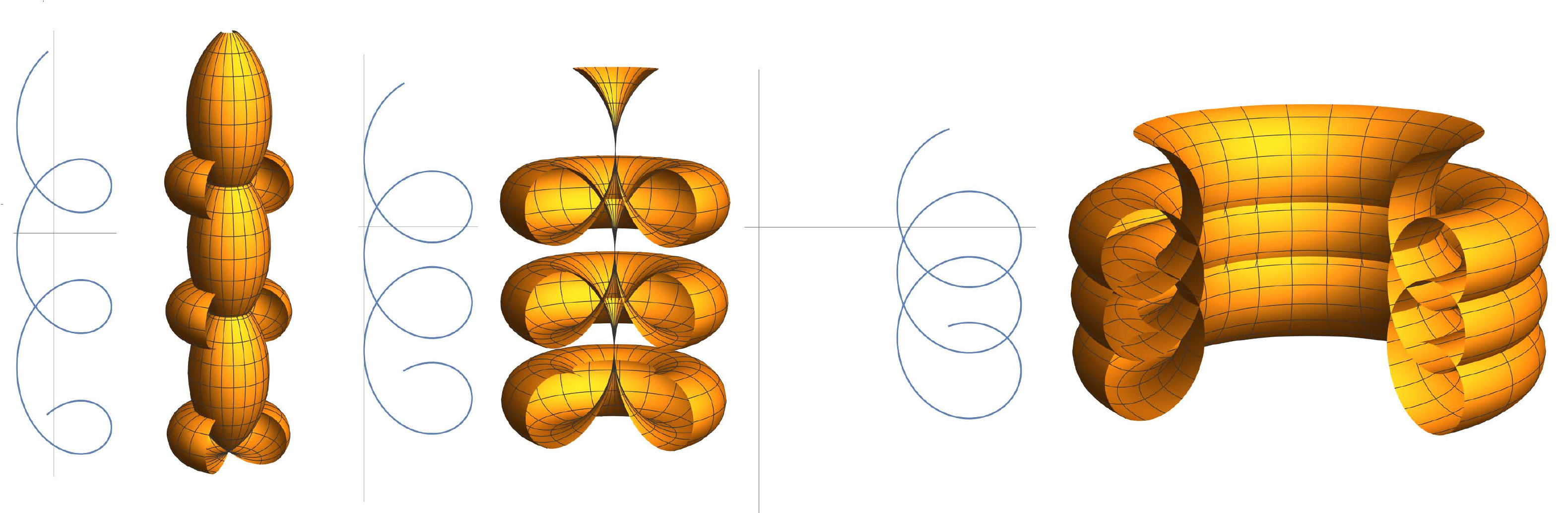}
		\caption{Open views of the loopoids $S_{\alpha_a^\eta}$, $a>0$, $\eta >0$; from left to right:
			$\cosh \eta < a+1$, $\cosh \eta = a+1$ and $\cosh \eta > a+1$ .}
		\label{fig:SuperExpMoreNeg1}
	\end{center}
\end{figure}
\end{itemize}

As a summary:
\begin{theorem}\label{Th:exp}
	The only rotational surfaces whose principal curvature on meridians is given by $k(x)=a\, e^x$, $a >0$, $x$ being the distance from the surface to the axis of revolution, are the following:
	\begin{enumerate}
		\item[(i)] the generalized catenoids of equal strength $S_{\alpha_a^\beta}$, $a>0$, $\beta \in (-\frac{\pi}{2},\frac{\pi}{2})$, see \eqref{eq:x beta} and \eqref{eq:z beta};
		\item[(ii)] the ondualysoid $S_{\alpha_a^{-1}}$, see \eqref{eq:x -1} and \eqref{eq:z -1};
		\item[(iii)] the loopoids $S_{\alpha_a^\eta}$, $a>0$, $\eta >0$, see \eqref{eq:x eta} and \eqref{eq:z eta}.
	\end{enumerate}
\end{theorem}


\section{Prescribing mean and Gaussian curvature on a rotational surface}\label{Sect4}

\subsection{Prescribing mean curvature on a rotational surface}\label{SectHprescribed}

In this section, we intend to study the problem of prescribing the mean curvature on a rotational surface in terms of a function $H=H(x)$ that depends on the distance $x$ from the surface to the axis of revolution. Using \eqref{eq:H}, we can interpret $\mathcal K'(x)+\mathcal K(x)/x=2H(x)$ as a linear o.d.e.\ with unknown the function $\mathcal K=\mathcal K(x)$. Then, Corollary \ref{cor:keysurfaces} leads to the following result:

\begin{theorem}\cite[Theorem 4.(a)]{CC22}
	\label{th:Prescribe_H}
 Let $H=H(x)$, $x\in D \subseteq \R $, be a continuous function. 
		Then there exists a one-parameter family of rotational surfaces with mean curvature $H(x)$, 
		$x$ being the (non constant) distance from the surface to the axis of revolution.
		The surfaces in the family are uniquely determined, up to translations along $z$-axis, by the geometric linear momenta of their generatrix curves given by
		\begin{equation}\label{eq:mom-H}
			x \, \mathcal K (x) =2 \! \int \! x  H(x) dx.
		\end{equation}
\end{theorem}
The parameter in the uniparametric family described in Theorem \ref{th:Prescribe_H} comes from the integration constant in \eqref{eq:mom-H}.

As a consequence of Theorem \ref{th:Prescribe_H}, we provided in \cite[Section 4]{CC22}  simple new proofs of some classical results concerning rotational surfaces like Euler's theorem (cf.\ \cite{E44}) about minimal ones (see \cite[Corollary 3]{CC22}) and Delaunay's theorem on constant mean curvature ones (see \cite[Corollary 4]{CC22}).  

As an illustration of new possibilities of Theorem \ref{th:Prescribe_H}, we deal with rotational surfaces such that the mean curvature is  inversely proportional to the distance to the axis of revolution. Following our notation, we can assume without restriction that  
\begin{equation}\label{eq:Hmux}
H(x)=\mu/x, \ \mu>0, \, x>0.
\end{equation}
We point out that the condition imposed on $H$ in \eqref{eq:Hmux} is quite natural in the sense that it is invariant under dilations of the surface and so it is necessary to consider all possible values of the positive constant $\mu$. Applying Theorem \ref{th:Prescribe_H}, putting \eqref{eq:Hmux} in \eqref{eq:mom-H}, we must control the generatrix curves whose geometric linear momentum is given by
\begin{equation}\label{eq:K_Hinverse}
	\mathcal K(x)=2\mu+c/x, \ \mu >0, \, c\in \R.
\end{equation}
We intend to obtain such curves as graphs $z=z_\mu^c(x)$ through \eqref{eq:truco}, getting that
\begin{equation}\label{eq:zmuc}
z_\mu^c(x)=\int \frac{(2\mu x+c)dx}{\sqrt{P_\mu^c(x)}}, \  P_\mu^c(x):=(1-4\mu^2)x^2-4\mu c\, x - c^2.
\end{equation}
The discriminant of the second degree polynomial $P_\mu^c(x)$ is $4c^2$.
We first observe from \eqref{eq:K_Hinverse} that if $c=0$, then $\mathcal K \equiv 2\mu$ and necessarily $\mu <1/2$, arriving so at the cone with opening $\theta_0 = \arcsin (2\mu)$ (see Section \ref{Sect2}). From \eqref{eq:zmuc}, we directly get that $z_\mu^0(x)=\frac{2\mu}{\sqrt{1-4\mu^2}}x=\tan \theta_0 \, x$, $x>0$, that it is just the half-straight line generatrix of the above mentioned cone. 

We assume, from now on, that $c\neq 0$. We distinguish the following cases taking into account that $P_\mu^c(x)>0$:
\begin{itemize}
	\item Case $\mu=1/2$: 
	
	Then $P_{1/2}^c(x)=-c(2x+c)$, with $c<0$ and $x>-c/2$. Using \eqref{eq:zmuc}, we deduce that
	\begin{equation}\label{eq:z05c}
		z_{1/2}^c(x)=\pm\frac{1}{3\sqrt{-c}}(x+2c)\sqrt{2x+c}.
	\end{equation}
See Figure \ref{fig:H05x surfaces}.
\begin{figure}[h!]
	\begin{center}
		\includegraphics[height=5.2cm]{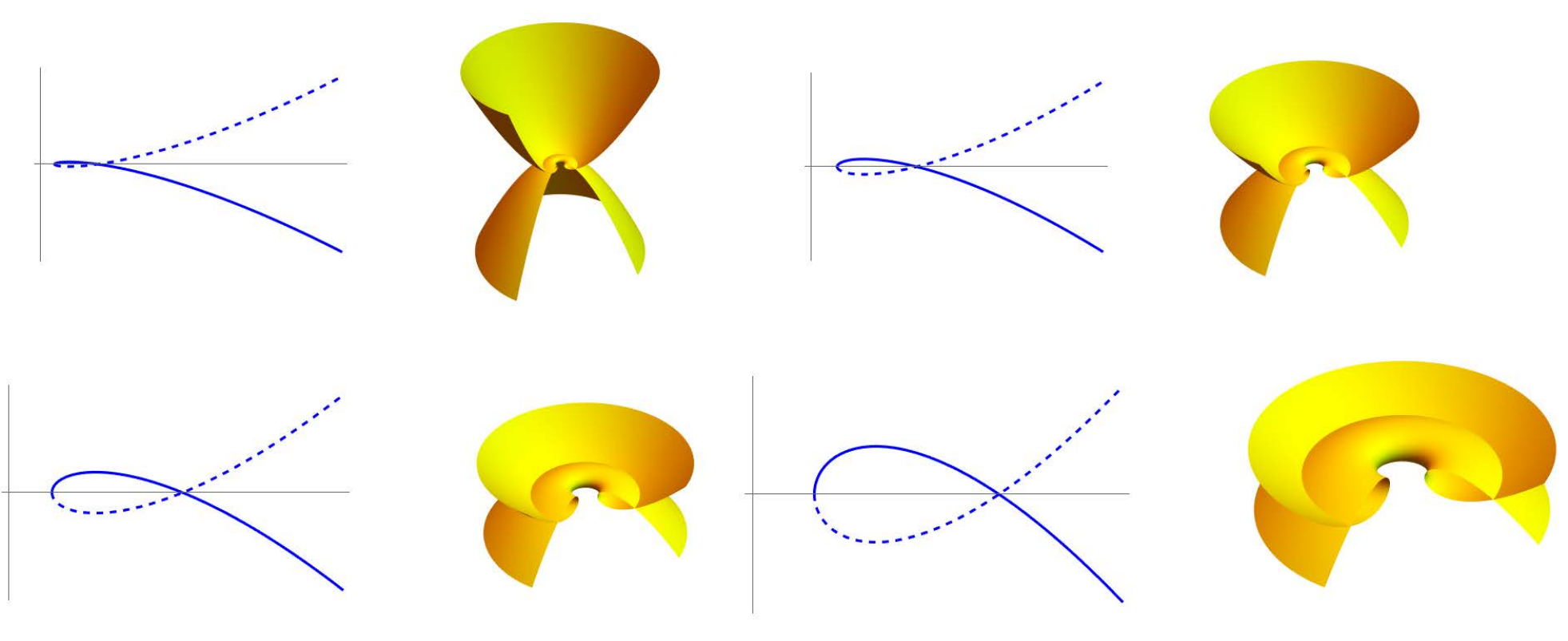}
		\caption{Open views of the rotational surfaces generated by $z=z_{1/2}^c(x)$, for different values of $c<0$.}
		\label{fig:H05x surfaces}
	\end{center}
\end{figure}
	
	\item Case $0<\mu:=\frac12 \sin \theta < 1/2$, $0<\theta<\pi/2$: 
	
	Now $P_\theta^c(x)=\cos^2 \theta \, x^2 -2 \sin \theta \ c \, x - c^2$, with $c\in \R$. If $c>0$, then $x>\frac{c}{1-\sin \theta}$; if $c<0$, then $x>\frac{-c}{1+\sin \theta}$. A straightforward long computation from \eqref{eq:zmuc} gives that
	\begin{equation}\label{eq:zthetac}
		z_\theta^c(x)=\frac{s_\theta}{c_\theta^2}\sqrt{P_\theta^c(x)}+\frac{c}{c_\theta^3} \ln \left(2 c_\theta \sqrt{P_\theta^c(x)} + 2 c_\theta ^2 x- 2 s_\theta c\right),
	\end{equation}
where $c_\theta:=\cos \theta$ and $s_\theta:=\sin \theta$. We remark that $z_\theta^0(x)=\tan \theta \, x$. See Figure \ref{fig:Hthetax surfaces}.

\begin{figure}[h!]
	\begin{center}
		\includegraphics[height=7cm]{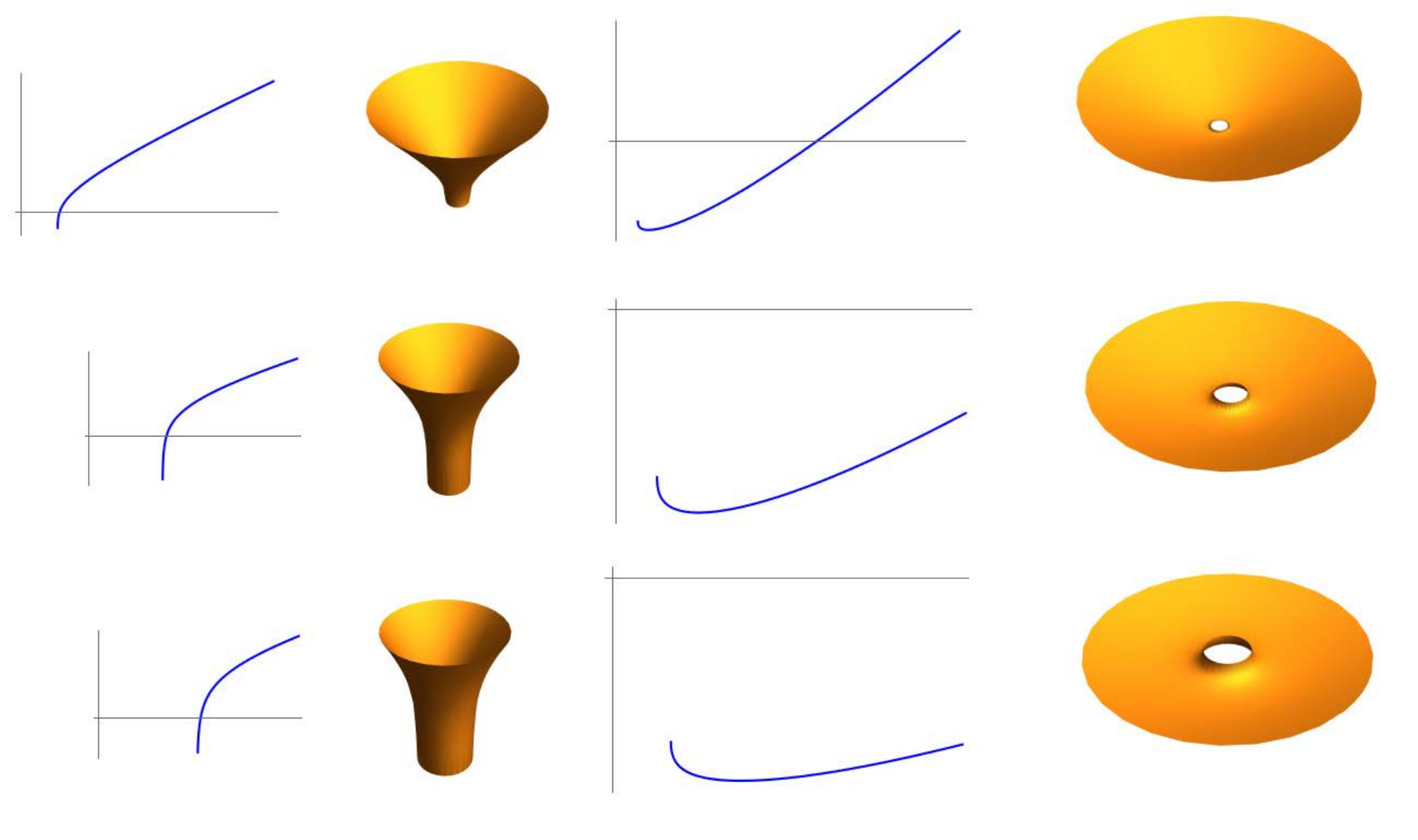}
		\caption{Rotational surfaces generated by $z=z_\theta^c(x)$, for a fixed values of $\theta$ ($\theta = \pi/4$) and different values of $c>0$ (on the left) and $c<0$ (on the right).}
		\label{fig:Hthetax surfaces}
	\end{center}
\end{figure}

	\item Case $\mu:=\frac12 \cosh \delta >1/2$, $\delta >0$: 
	
	In this case, $P_\delta^c(x)=- \sinh^2 \delta x^2 - 2 \cosh \delta \, c \, x - c^2$, with $c<0$ and $\frac{-c}{1+\cosh \delta} < x < \frac{c}{1-\cosh \delta}$. Another straightforward long computation from \eqref{eq:zmuc} leads to 
	\begin{equation}\label{eq:zdeltac}
		z_{\delta}^c(x)=-\frac{ch_\delta}{sh^2_\delta}\sqrt{P_\delta^c(x)}+\frac{c}{sh^3_\delta}\arcsin \left( \frac{sh^2 \delta x + ch_\delta c}{c}\right),
	\end{equation}
where $ch_\delta:=\cosh \delta$ and $sh_\delta:=\sinh \delta$. See Figure \ref{fig:Hdeltax surfaces}.

\begin{figure}[h!]
	\begin{center}
		\includegraphics[height=8cm]{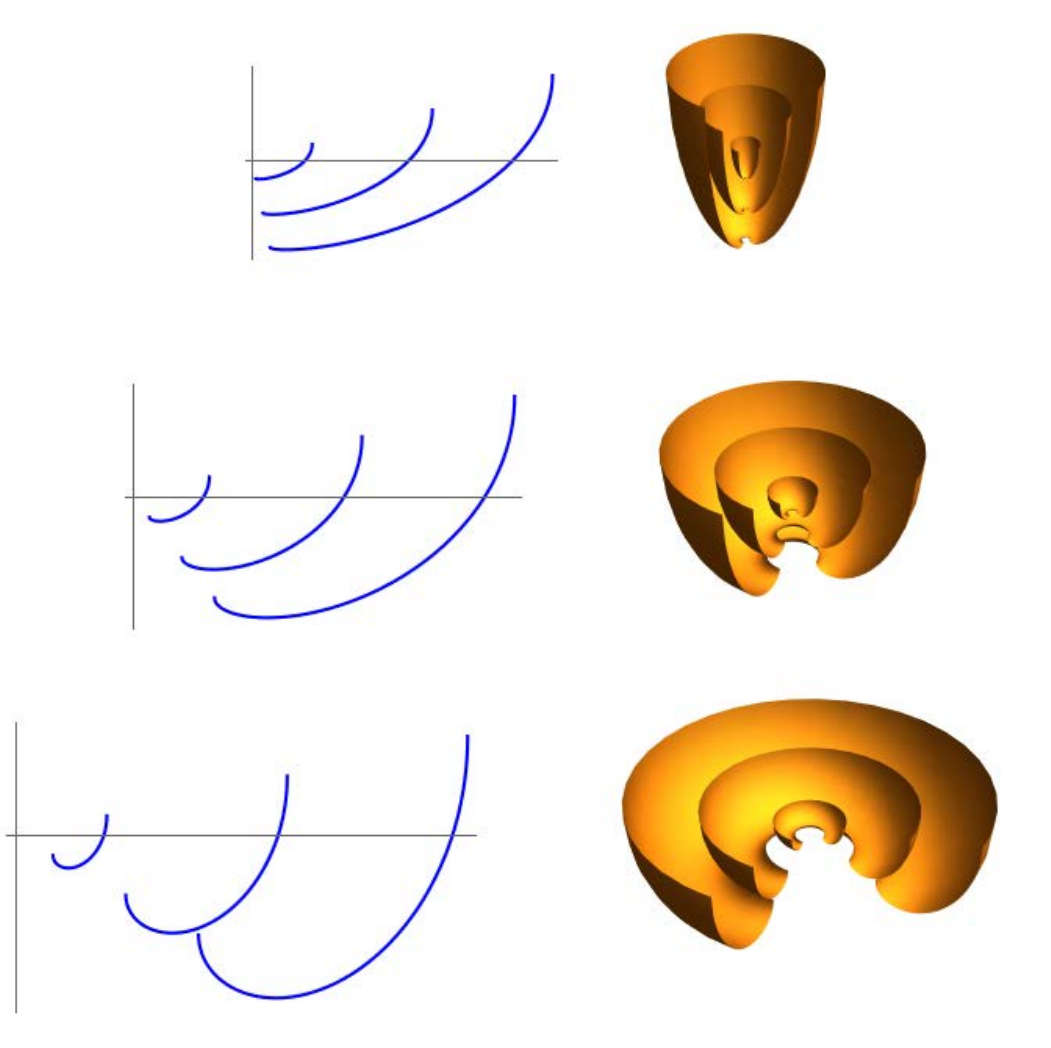}
		\caption{Open views of the rotational surfaces generated by $z=z_\delta^c(x)$, for different values of $\delta >0$ and different values of $c<0$.}
		\label{fig:Hdeltax surfaces}
	\end{center}
\end{figure}

\end{itemize}

As a summary, we collect in the next result the classification of rotational surfaces with mean curvature inversely proportional to the distance from the surface to the axis of revolution.

\begin{theorem}\label{th:Class_Hinvx}
	Let $S_z$ be a rotational surface written as
$$ S_z \equiv X(x,\vartheta)=\left(x  \cos \vartheta, x \sin \vartheta, z(x) \right), \
 x>0, \,  \vartheta \in (-\pi,\pi).$$
 If the mean curvature $H$ of $S_z$ is given by $H(x)=\mu/x$, $\mu>0$, there are three possibilities:
 \begin{itemize}
 	\item[(i)] $\mu=\frac12$ and $z=z_{1/2}^c(x)$, $x>-\frac{c}{2}$, with $c<0$, see \eqref{eq:z05c}.
 	\item[(ii)]  $0<\mu:=\frac12 \sin \theta < \frac12$, $0<\theta<\frac{\pi}{2}$ and $z=z_\theta^c(x)$, $x>\frac{c}{1-\sin \theta}$ (resp.\ $x>\frac{-c}{1+\sin \theta}$) if $c>0$ (resp.\ if $c<0$), see \eqref{eq:zthetac}; $z=z_\theta^0(x)=\tan \theta \, x$, if $c=0$ (i.e.\ the circular cone with opening $\theta$).
 	\item[(iii)] $\mu:=\frac12 \cosh \delta >\frac12$, $\delta >0$ and $z=z_{\delta}^c(x)$, $\frac{-c}{1+\cosh \delta} < x < \frac{c}{1-\cosh \delta}$,  with $c<0$, see \eqref{eq:zdeltac}.
 \end{itemize}
\end{theorem}



\subsection{Prescribing Gaussian curvature on a rotational surface}\label{SectKGprescribed}

In this section, we aim to study the problem of prescribing the Gauss curvature on a rotational surface in terms of a function $K_G=K_G(x)$ that depends on the distance $x$ from the surface to the axis of revolution. Using \eqref{eq:KGauss}, we can interpret $\mathcal K'(x)\mathcal K(x)/x=K_G(x)$ as a trivial o.d.e.\ with unknown the function $\mathcal K=\mathcal K(x)$. Then, Corollary \ref{cor:keysurfaces} implies the following result:

\begin{theorem}\cite[Theorem 4.(b)]{CC22}
	\label{th:Prescribe_KGauss}
 Let $K_{\text{G}}=K_{\text{G}}(x)$, $x\in D \subseteq \R $, be a continuous function. 
		Then there exists a one-parameter family of rotational surfaces with Gauss curvature  $K_{\text{G}}(z)$,
		$x$ being the (non constant) distance from the surface to the axis of revolution. The surfaces in the family are uniquely determined, up to translations along $z$-axis, by the geometric linear momenta of their generatrix curves given by
		\begin{equation}\label{eq:mom-KGauss}
			\mathcal K(x)^2 =2 \! \int \! x  K_{\text G}(x) dx.
		\end{equation}
\end{theorem}
The parameter in the uniparametric family described in Theorem \ref{th:Prescribe_KGauss} comes from the integration constant in \eqref{eq:mom-KGauss}.

As a consequence of Theorem \ref{th:Prescribe_KGauss}, we provided in \cite{CC22} a simple new proof of Darboux's theorem concerning rotational surfaces with constant Gauss curvature (see \cite[Corollary 5]{CC22}).


As an illustration of some new applications of Theorem \ref{th:Prescribe_KGauss}, we deal with rotational surfaces such that the Gauss curvature is  inversely proportional to the distance to the axis of revolution. Following our notation, we can assume without restriction that 
\begin{equation}\label{eq:KGinvx}
K_G(x)=\mu/x, \ \mu >0, \, x\neq 0.
\end{equation}
Applying Theorem \ref{th:Prescribe_KGauss}, putting \eqref{eq:KGinvx} in \eqref{eq:mom-KGauss}, we must control the generatrix curves whose geometric linear momentum is given by
\begin{equation}\label{eq:K_KGinvx}
	\mathcal K(x)=\sqrt{2\mu x+c}, \,  c\in\R
\end{equation}
Now we recall (see Section \ref{Sect2}) that the onducycloid of radius $R>0$, defined as the surface generated by the rotation of the cycloid around its base, was determined by its geometric linear momentum $\mathcal K(x)=\sqrt x / \sqrt{2R}$. So the case $c=0$ in \eqref{eq:K_KGinvx} corresponds to the onducycloid of radius $R=\frac{1}{4\mu}$.

In the general case $c\neq 0$, looking at \eqref{eq:K_KGinvx}, we can make now the following simpler reasoning.
We translate in the $x$-direction the generatrix cycloid of radius $R$ of the onducycloid (corresponding to $a=0$ in the following) considering the family of cycloids $C_a^R$, $a\in \R$, given by $x=a+R(1-\cos t)$, $z=R(t-\sin t -\pi)$, $t\in [0,2\pi]$ (see Figure \ref{fig:Cycloids}.)

\begin{figure}[h!]
	\begin{center}
		\includegraphics[height=5cm]{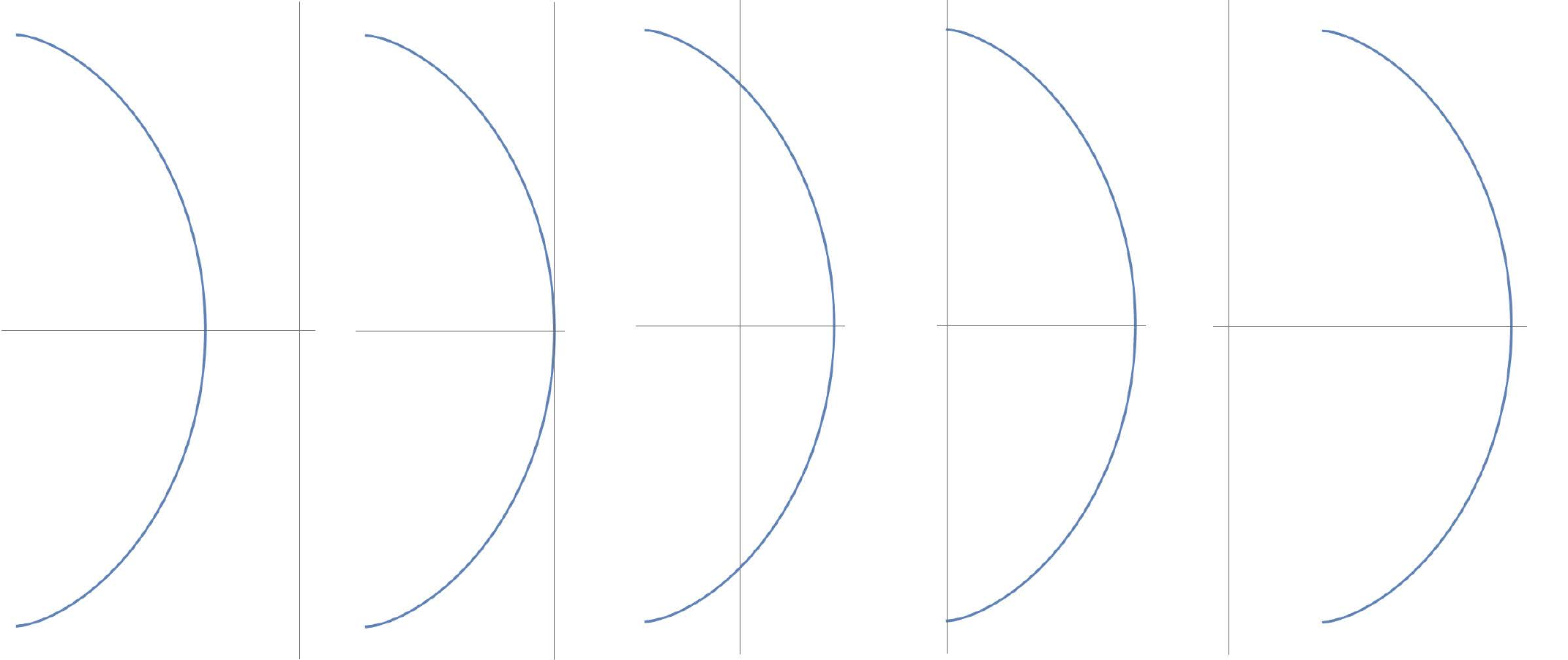}
		\caption{Cycloids $C_a^R$ with a fixed radius $R>0$ and (from left to right) $a<-2R$, $a=-2R$, $-2R<a<0$, $a=0$, $a>0$.}
		\label{fig:Cycloids}
	\end{center}
\end{figure}

Using Remark \ref{re:no ppa}, it is an exercise to check that the geometric linear momentum determining the cycloids $C_a^R$ is given by
\begin{equation}\label{eq:K_cycloids}
	\mathcal K(x)=\sqrt{\frac{x-a}{2R}}, \ x \geq a.
\end{equation}
Comparing \eqref{eq:K_KGinvx} and \eqref{eq:K_cycloids}, we easily deduce that the rotational surfaces we were looking for are nothing but the generated by the cycloids $C_a^R$, with $a=-2Rc$ and, of course, $R=\frac{1}{4\mu}$. We will call them \textit{transonducycloids} of radius $R$, since the case $a=0 \Leftrightarrow c=0$ recovers the onducycloid of radius $R$ (see Figure \ref{fig:transonducycloidsycloids}). 

\begin{figure}[h!]
	\begin{center}
		\includegraphics[height=4cm]{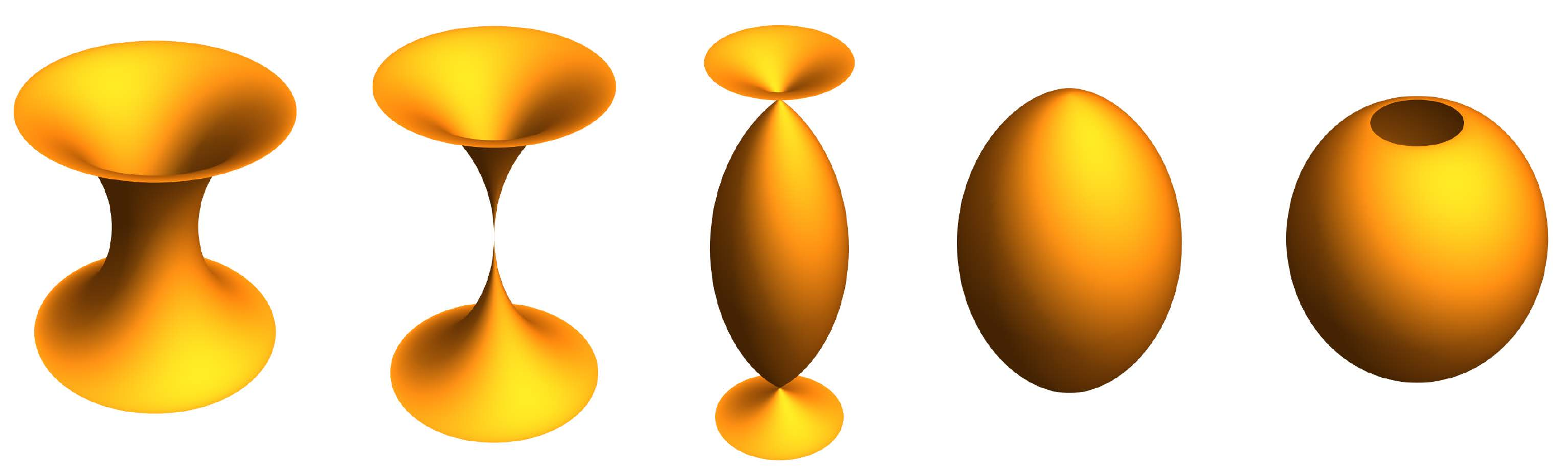}
		\caption{Transonducycloids $S_{C_a^R}$ with a fixed radius $R>0$ and (from left to right) $a<-2R$, $a=-2R$, $-2R<a<0$, $a=0$, $a>0$.}
		\label{fig:transonducycloidsycloids}
	\end{center}
\end{figure}

As a summary, we compile in the next result the classification of rotational surfaces with Gauss curvature inversely proportional to the distance from the surface to the axis of revolution.

\begin{theorem}\label{th:Class KGinverse}
	The transonducycloids $S_{C_a^R}$, defined as the surfaces of revolution generated by the rotation of the cycloids $C_a^R$ of radius $R>0$ around its base after an $a$-translation, $a\in \R$, in the orthogonal direction to its base, are the only rotational surfaces with Gauss curvature $K_G(x)=\mu/x$, $ \mu >0$, $ 0\neq x$ being the distance from the surface to the axis of revolution and $\mu=\frac{1}{4R}$.
\end{theorem}


\subsection{Constraint between the mean curvature and the Gaussian curvature on a rotational surface.}

An interesting consequence of combining Theorems \ref{th:Prescribe_H} and \ref{th:Prescribe_KGauss} is that the mean and Gaussian curvatures on a rotational surface cannot be prefixed simultaneously in terms of arbitrary functions depending on the distance from the surface to the axis of revolution. We now prove that they are necessarily related. Following our notation, suppose that we prescribe continuous functions $H=H(x)$ and $K_{\text{G}}=K_{\text{G}}(x)$, $x\in D \subseteq \R $. Then, using \eqref{eq:mom-H} and \eqref{eq:mom-KGauss}, we arrive at a double determination of the geometric linear momentum characterizing the rotational surface:
\begin{equation}\label{eq:mom-both}
	x\mathcal{K}(x)=2\int xH(x)dx, \quad 	\mathcal{K}(x)^2=2\int xK_G(x)dx.
\end{equation}
Eliminating $\mathcal K = \mathcal K(x)$ in \eqref{eq:mom-both}, we deduce the constraint given by
\begin{equation}\label{eq:HlinkKG}
\left(\int x H(x)dx\right)^2=\dfrac{x^2}{2}\int x K_G(x)dx
\end{equation}
Derivating with respect to $x$ in \eqref{eq:HlinkKG}, we conclude the following result.
\begin{proposition}\label{propo:KG from H}
Let $S_{(x,z)}$ be a rotational surface with mean curvature given by a continuous function $H=H(x)$, $x\in D \subseteq \R $. Then the Gauss curvature of $S_{(x,z)}$ is given by
\begin{equation}\label{eq:KG one parameter}
	K_G(x)=\dfrac{2}{x}\dfrac{d}{d x}\left(\dfrac{1}{x^2}\left(\int x H(x)dx\right)^2\right).
\end{equation}
\end{proposition}
\begin{remark}\label{re:KG uniparametric}
	An immediate consequence of Proposition \ref{propo:KG from H} is that once prescribed $H=H(x)$ for a rotational surface $S_{(x,z)}$, its Gauss curvature $K_G=K_G(x)$ belongs to a one-parameter family of functions parameterized by the integration constant of the integral in \eqref{eq:KG one parameter}. For example, a simple calculus in \eqref{eq:KG one parameter} leads to the the Gauss curvature of the rotational $H_0$-surfaces, i.e.\ surfaces with constant mean curvature $H_0\neq 0$, given by $K_G(x)=H_0^2-4\Gamma^2/x^4$, $\Gamma\in \R$. The round sphere is the surface with $\Gamma =0$ in the family of rotational $H_0$-surfaces.
\end{remark}
Inspired by the family of rotational $H_0$-surfaces, bearing in mind Remark \ref{re:KG uniparametric}, we introduced the notion of \textit{Gaussian constant} for an interesting class of rotational surfaces generalizing $H_0$-surfaces.
\begin{definition}\label{def:Gaussian constant}
	Let $S_{(x,z)}$ be a rotational surface with mean curvature $H$ given by
	$$H(x)=\mu \, x^n, \, x>0, \quad \mu >0,  \, n \neq -2. $$ 
	The \textit{Gaussian constant} $\Gamma\in \R$ is defined by the equality
	$$
	\int x H(x)dx =\frac{\mu}{n+2}x^{n+2} + \Gamma.
	$$
\end{definition}
Making use of Proposition \ref{propo:KG from H} and the concept of Gaussian constant introduced in Definition \ref{def:Gaussian constant}, we provide the next uniqueness result for the Hopf-Kühnel surfaces defined in Definition \ref{def:Kuhnel}.

\begin{theorem}\label{Th:H KG mononomial}
	Let $S_{(x,z)}$ be a rotational surface with mean curvature $H$ given by 
	\begin{equation}\label{eq:H monomial}
		H(x)=\mu \, x^n, \, x>0, \quad \mu >0,  \, n \neq -2.
	\end{equation}
Then its Gauss curvature is given by
\begin{equation}\label{eq:KG monomial}
K_G(x) = \dfrac{4(n+1)\mu^2}{(n+2)^2} \, x^{2n} + \dfrac{4n  \Gamma  \mu }{n+2}\, x^{n-2}
-	\dfrac{4\Gamma^2}{x^4},
\end{equation}
with $\Gamma \in \R$ the Gaussian constant (see Definition \ref{def:Gaussian constant}.)
Moreover, the only rotational surface with mean curvature given by \eqref{eq:H monomial} and null Gaussian constant, i.e.\ $\Gamma =0$ in \eqref{eq:KG monomial}, is the Hopf-Kühnel surface $S_{\alpha_q}$, with $q=n+1$, when $n\neq -1$, and the cone with opening $\arcsin 2\mu$ when $n=-1$.
\end{theorem}

\begin{proof}
Using Definition \ref{def:Gaussian constant} and putting \eqref{eq:H monomial} in \eqref{eq:KG one parameter}, we easily reach \eqref{eq:KG monomial}. On the other hand, Theorem \ref{th:Prescribe_H} implies that $S_{(x,z)}$ is determined  by the geometric linear momentum
\begin{equation}\label{eq:mom H monomial}
\mathcal K (x)=\dfrac{2\mu}{n+2}\, x^{n+1}+ \dfrac{c}{x}, \  c\in \R.
\end{equation}
Now, using \eqref{eq:mom H monomial} in \eqref{eq:KGauss}, we arrive again at \eqref{eq:KG monomial} considering $c=2\Gamma$.
In this way, if we take $\Gamma=0=c$, then Corollary \ref{cor:keysurfaces} and \eqref{eq:mom H monomial} says  to us that $S_{(x,z)}$ is uniquely determined, up to $z$-translations, by $\mathcal K (x)=\frac{2\mu}{n+2}\, x^{n+1}$. Looking at \eqref{eq:K Kuhnel}, this corresponds to the Hopf-Kühnel surface $S_{\alpha_q}$, with $q=n+1$ provided that $n+1\neq 0$. If $n+1=0$, we arrive at $\mathcal K =2\mu$ and this leads to the cone with opening $\arcsin 2\mu$ (see Section \ref{Sect2}).
\end{proof}

\begin{remark}
Recall that we completely classify in Theorem \ref{th:Class_Hinvx} the rotational surfaces corresponding to the case $n=-1$	in Theorem \ref{Th:H KG mononomial}.  The Gauss curvature of these surfaces is, using that $c=2\Gamma$,  given by 
$
K_G(x) = - 2c \mu/ x^3 -c^2/x^4. 
$
\end{remark}

As direct consequences of Theorem \ref{Th:H KG mononomial}, looking at Section \ref{Sect2}, we can deduce the following uniqueness results for the Myllar balloons ($q=2 \Leftrightarrow n=1$), the onducycloids ($q=1/2 \Leftrightarrow n=-1/2$) and the Flamm's paraboloids ($q=-1/2 \Leftrightarrow n=-3/2$).

\begin{corollary}\label{cor:Mylar}
	The Mylar balloons are the only rotational surfaces with mean curvature $H(x)=\mu \, x$, $\mu >0$,
	and Gauss curvature $K_G (x)=\lambda \, x^2$, $\lambda \neq 0$, $x$ being the distance from the surface to the axis of revolution. In addition, necessarily $\lambda=8 \mu^2/9= 2/r^2$, with $r$ the inflated radius of the balloon.
\end{corollary}

\begin{corollary}\label{cor:Onducycloids}
	The onducycloids are the only rotational surfaces with mean curvature $H(x)=\mu /\sqrt  x$, $\mu >0$,
	and Gauss curvature $K_G (x)=\lambda / x$, $\lambda \neq 0$, $x$ being the distance from the surface to the axis of revolution. In addition, necessarily $\lambda=\frac89 \mu^2= \frac{1}{4R}$, with $R$ the radius of the anticycloid.
\end{corollary}

\begin{corollary}\label{cor:Flamm paraboloids}
	The Flamm's paraboloids are the only rotational surfaces with mean curvature $H(x)=\frac{\mu}{x\sqrt  x}$, $\mu >0$,
	and Gauss curvature $K_G (x)=\lambda / x^3$, $\lambda \neq 0$, $x$ being the distance from the surface to the axis of revolution. In addition, necessarily $\lambda=-8 \mu^2= -r_S /2$, with $r_S$ the Schwarzchild radius of the paraboloid.
\end{corollary}

Using the same technique, we finish providing a new characterization of the tori of revolution.
Recall from Section \ref{Sect2} that the torus of revolution with major radius $|a|\neq 0$ and minor radius $R>0$, given by $(\sqrt{x_1^2+x_2^2}-a)^2+x_3^2=R^2$, is uniquely determined by the geometric linear momentum  $\mathcal K (x)=(x-a)/R$.
Then \eqref{eq:H} and \eqref{eq:KGauss} imply that
$H(x)=\frac{1}{R}-\frac{a}{2Rx}$ and $K_G(x)=\frac{1}{R^2}-\frac{a}{R^2x}$.

\begin{corollary}\label{cor:tori}
	The tori of revolution are the only rotational surfaces with mean curvature $H(x)=\lambda +\mu / x$, $\lambda >0$, $\mu \neq 0$,
	and Gauss curvature $K_G (x)=\gamma + \nu /x$, $\gamma >0$ $\nu \neq 0$, $x$ being the distance from the surface to the axis of revolution. In addition, necessarily $\gamma=\lambda^2= 1/R^2$, with $R$ the minor radius of the torus, and $\nu=2 \lambda \mu= -a/R^2$, with $a$ the major radius of the torus.	
\end{corollary}



\begin{thebibliography}{1}\bibliographystyle{alpha}


\bibitem{BGM20a} A.~Bueno, J.A.~G\'alvez and P.~Mira. {\em Rotational hypersurfaces of prescribed mean curvature.}  J.\ Differential Equations \textbf{268} (2020), 2394--2413.

\bibitem{BGM20b} A.~Bueno, J.A.~G\'alvez and P.~Mira. {\em The global geometry of surfaces with prescribed mean curvature in $\R^3$.}  Trans.\ Amer.\ Math.\ Soc.\ \textbf{373} (2020), 4437--4467.

\bibitem{BO22} A.~Bueno, and I.~Ortiz. {\em Rotational surfaces of prescribed Gauss curvature in $\R^3$.}   arXiv preprint arXiv:2201.07057, 2022 - arxiv.org.

\bibitem{BO23} A.~Bueno, and I.~Ortiz. {\em Surfaces of prescribed linear Weingarten curvature in $\R^3$.}   Proc.\ Roy.\ Soc.\ Edinburgh Sect.\ A \textbf{153}(4) (2023), 1347--1370.

\bibitem{CC22} P.~Carretero and I.~Castro. {\em A new approach to rotational Weingarten surfaces.} 
Mathematics {\bf 2022} 10(4), 578; https://doi.org/10.3390/math10040578

\bibitem{CCI16} I.~Castro and I.~Castro-Infantes.
{\em Plane curves with curvature depending on distance to a line.}
Diff.\ Geom.\ Appl.\ {\bf 44} (2016), 77--97.

\bibitem{CCIs20} I.~Castro, I.~Castro-Infantes and J.~Castro-Infantes.
{\em On a Problem of David Singer about Prescribing Curvature for Curves.}
Geom.\ Integrability \& Quantization {\bf 21} (2020), 100--117.

\bibitem{Ch45} S.S~Chern.
{\em Some new characterizations of the Euclidean sphere.}
Duke Math.\ J.\ {\bf 12} (1945), 279--290.

\bibitem{D90} G.~Darboux.
{\em Sur la surface dont la courbure totale est constante}.
Ann.\ Sci.\ \'Ec.\ Norm.\ Sup\'er.\ (3) {\bf 7} (1890), 9--18.

\bibitem{D41} C.~Delaunay.
{\em Sur la surface de revolution dont la courbure moyenne est constante}.
J.\ Math.\ Pures Appl.\ {\bf 6} (1841), 309--320.

\bibitem{E44} L.~Euler.
{\em Methodus inveniendi lineas curvas: maximi minimive proprietate gaudentessive solutio problematis isoperimetrici latissimo sensu accepti (in Latin)}. 
Opera Omnia: Series 1, Volume \textbf{24} (1744).

\bibitem{F93} R.~Ferr\'{e}ol.
{\em Encyclop\'{e}die des formes math\'{e}matiques remarquables.}
www.\ mathcurve.\ com

\bibitem{H51} H.~Hopf.
{\em \"{U}ber Fl\"{a}chen mit einer Relation zwischen den Hauptkr\"{ummungen}.}
Math.\ Nachr.\ {\bf 4} (1951), 232--249.

\bibitem{Ke80} K.~Kenmotsu.
{\em Surfaces of revolution with prescribed mean curvature.}
T\^{o}hoku  Math.\ J.\ {\bf 32} (1980), 147--153.

\bibitem{Kr60} M.D.~Kruskal.
{\em Maximal extension of Schwarzschild metric}. Physical Reviews \textbf{119}(5) (1960), 1743--1745.

\bibitem{K15} W.~Kühnel.
{\em Differential Geometry, Curves-Surfaces-Manifolds.}
AMS, 2015.

\bibitem{KS05} W.~Kühnel and M.~Steller.
{\em On closed Weingarten surfaces.}
Monatsh.\ Math.\ {\bf 146} (2005), 113--126.

\bibitem{LP20} R.~L\'opez and A.~P\'ampano.
{\em Classification of rotational surfaces in Euclidean space satisfying a lineal relation between their principal curvatures}. 
Math.\ Nachr.\  \textbf{293}(4) (2020), 735--753.

\bibitem{MO03} I.V.~Mladenov and J.~Oprea.
{\em The Mylar balloon revisited.}
Amer.\ Math.\ Monthly {\bf 110} (2003), 761--784.

\bibitem{P94} W.~Paulsen.
{\em What is the shape of a Mylar balloon?} 
Amer.\ Math.\ Monthly \textbf{101} (1994), 953--958.

\bibitem{S99} D.~Singer.
{\em Curves whose curvature depends on distance from the origin.}
Amer.\ Math.\ Monthly {\bf 106} (1999), 835--841.

\bibitem{S08} D.~Singer.
{\em Lectures on elastic curves and rods.}
Curvature and Variational Modeling in Physics and Biophysics,
AIP Conf.\ Proc.\ vol.\ {\bf 1002} (2008), 3--32.

\bibitem{W61} J.~Weingarten. 
{\em Ueber eine Klasse auf einander abwickelbarer Fl\"{a}chen}. 
J.\ Reine Angew.\ Math.\, \textbf{59} (1861) 382--393.

\end{thebibliography}
\end{document}